\documentclass{article}
\usepackage{amssymb,amsmath,amsthm,graphicx}
\usepackage[all]{xy}
\def\qed{\hfill {\hbox{${\vcenter{\vbox{               
   \hrule height 0.4pt\hbox{\vrule width 0.4pt height 6pt
   \kern5pt\vrule width 0.4pt}\hrule height 0.4pt}}}$}}}

\def\utr{\underline{\ast}}
\def\otr{\overline{\ast}}

\newtheorem{theorem}{Theorem}

\theoremstyle{definition}
\newtheorem{example}{Example}
\newtheorem{definition}{Definition}
\newtheorem{remark}{Remark}

\date{}

\title{\Large \textbf{Bikei Module Invariants of Unoriented Surface-Links}}

\author{Yewon Joung\footnote{Email: yewon112@gmail.com The first  author was supported by Basic Science Research Program through the National Research Foundation of Korea (NRF) funded by the Ministry of Education, Science and Technology (NRF-2019R1F1A1060205).}\and Sam Nelson\footnote{Email: Sam.Nelson@cmc.edu. Partially supported by Simons Foundation collaboration grant 702597.}\\
}

\begin{document}
\maketitle

\begin{abstract}
We extend our previous work from \cite{JN1} on biquandle module invariants of 
oriented surface-links to the case of unoriented surface-links using \textit{bikei 
modules}. The resulting infinite family of enhanced invariants proves be
effective at distinguishing unoriented and especially non-orientable 
surface-links; in particular, we show that these invariants are more 
effective than the bikei homset cardinality invariant alone at distinguishing 
non-orientable surface-links. Moreover, as another application we note that 
our previous biquandle modules which do not satisfy the bikei module axioms 
are capable of distinguishing different choices of orientation for orientable 
surface-links as well as classical and virtual links.
\end{abstract}

\parbox{5.5in} {\textsc{Keywords:} surface-links, marked graph diagrams, 
bikei, enhancements

\smallskip

}

\section{\large\textbf{Introduction}}\label{I}

\textit{Surface-links} are compact surfaces which may be knotted and linked in 
$\mathbb{R}^4$. Surface-links may be represented combinatorially 
via several diagrammatic systems such as  \textit{broken surface diagrams} 
drawn as immersions of surfaces into $\mathbb{R}^3$ with breaks to indicate 
crossing information, graphs representing the singular set of the embedding 
enhanced with relative height information known as \textit{braid charts},
and ordered sequences of link diagrams representing horizontal cross-sections
of the surface-link with respect to a choice of vertical direction known as
\textit{movie diagrams}. If we isotope the surface-link to have all of its
maxima at one level, minima at another and saddle points at the same 
intermediate level, the cross-section at the saddle-level is a knotted 
4-valent graph whose vertices represent the saddle points. Marking these to 
indicate the direction of the saddle encodes all the information necessary to
recover the full surface-link, as shown by Lomanaco in \cite{L}. Such a diagram
is known as a \textit{marked vertex diagram}, \textit{ch-diagram} or 
\textit{marked graph diagram.} 
The combinatorial moves on marked graph diagrams encoding ambient isotopy
in $\mathbb{R}^4$ are known as the \textit{Yoshikawa moves}; see e.g. 
\cite{KJL, KJL2} for more.

\textit{Bikei} are algebraic structures with axioms encoding the 
Reidemeister moves for unoriented knots and links. Bikei form a special
case of \textit{biquandles} on the one hand and generalize \textit{kei},
also known as \textit{involutory quandles}, on the other. Bikei were
introduced in \cite{AN}; our notation in this paper comes from \cite{EN}.

\textit{Biquandle modules} are algebraic structures associated to biquandles
in which each order pair of biquandle elements determines an Alexander 
biquandle-style pair of operations on a commutative ring $R$ with identity.
The axioms are chosen so that a biquandle module over a biquandle $X$
associates an invariant $R$-module to each biquandle homomorphism 
$f:\mathcal{B}(L)\to X$ from the fundamental biquandle of an oriented link 
$L$ to a finite target biquandle $X$. The multiset of these modules is then
an enhanced invariant whose cardinality determines the biquandle counting 
invariant but is in general a stronger invariant. See \cite{JN1} for more.

In \cite{JN1} the authors used biquandle modules to define invariants of 
oriented surface-links. In particular, the use of orientation is critical for 
the definition of these invariants. For classical knots and links this is not a 
problem since every classical link has at least one orientation. However, 
surface-links include both orientable and non-orientable cases. 
In this paper we adapt the biquandle module idea to the case of non-orientable
surface-links by defining a notion of \textit{bikei modules}, the unoriented
version of biquandle modules.

The paper is organized as follows. In Section \ref{KOS} we review the basics
of surface-link theory and marked graph diagrams. In Section \ref{B} we recall
bikei and bikei colorings of surface-links. We then define the notion of bikei
modules and give some examples. In Section \ref{BE} we use bikei modules to 
define enhanced invariants of unoriented surface-links and provide some 
computational illustrations and examples. We conclude in Section \ref{Q} with
some questions for future research.

\section{\large\textbf{Surface-Links}}\label{KOS}

We begin with a definition.

\begin{definition}
A \textit{surface-link} is a closed 2-manifold smoothly (or piecewise 
linearly and locally flatly) embedded in the 4-space $\mathbb R^4$.
Two surface-links are \textit{equivalent} if they are ambient isotopic. 
A \textit{non-oriented surface-link} is a non-orientable surface-link 
or an orientable surface-link without orientation, while an \textit{oriented 
surface-link} is an orientable surface-link with a fixed orientation.
\end{definition}

We can specify non-oriented surface-links using \textit{marked graph diagrams},
unoriented 4-valent spatial graphs with vertices marked with a small bar, in
the following way. At each vertex we smooth the vertices in both ways to obtain
two link diagrams connected by a surface (a cobordism between the links) 
with saddle points at the marked vertices. 

\[\includegraphics{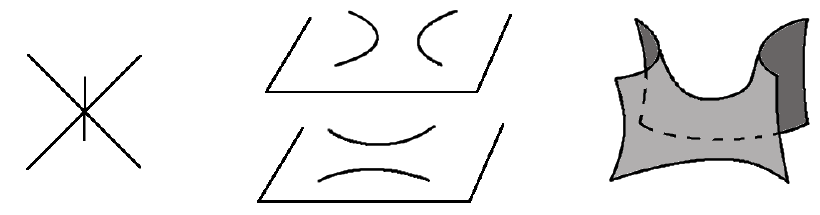}\] 

If the links are unlinks, we can
cap them off with disks to obtain a closed surface-link; such marked graph 
diagrams are called \textit{admissible}. Diagrams which are not admissible
represent cobordisms between the links represented by the smoothed diagrams.

\begin{example}
The marked graph diagram
\[\includegraphics{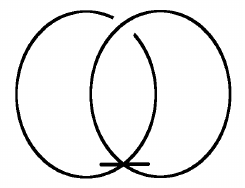}\] 
represents an unknotted real projective plane. Smoothing the marked vertex
in both ways and connecting yields a cobordism
\[\includegraphics{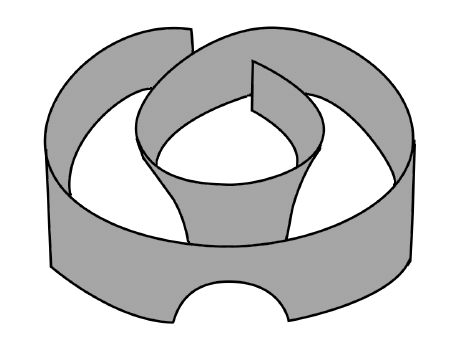}\] 
between two unknots with a single crossing each; these can be capped off with 
Whitney's umbrellas to complete the projective plane.
\end{example}

The moves on marked graph diagrams encoding ambient isotopy of the
represented surface are known as \textit{Yoshikawa moves}:
\[\scalebox{0.9}{\raisebox{0.4in}{\includegraphics{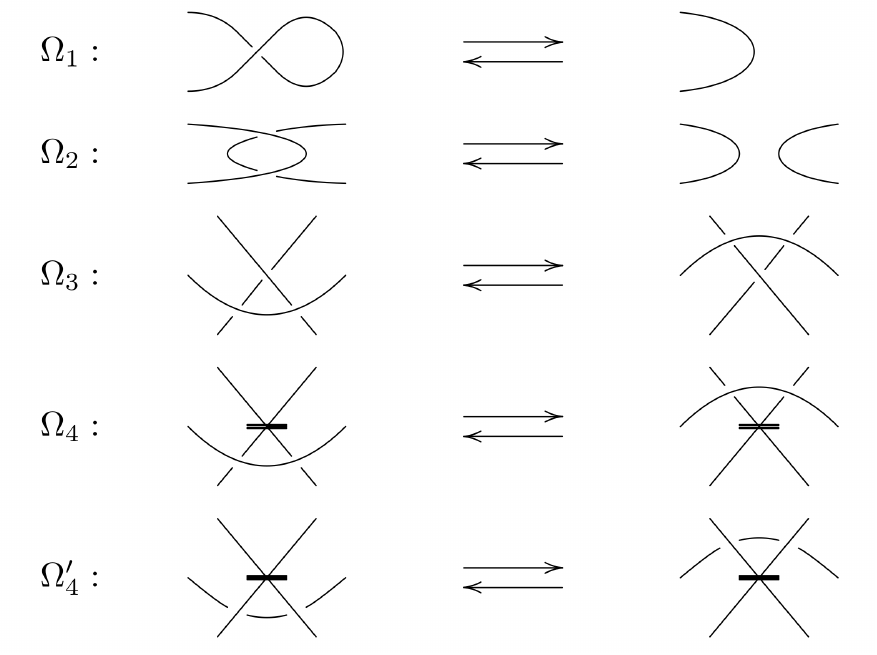}}
\includegraphics{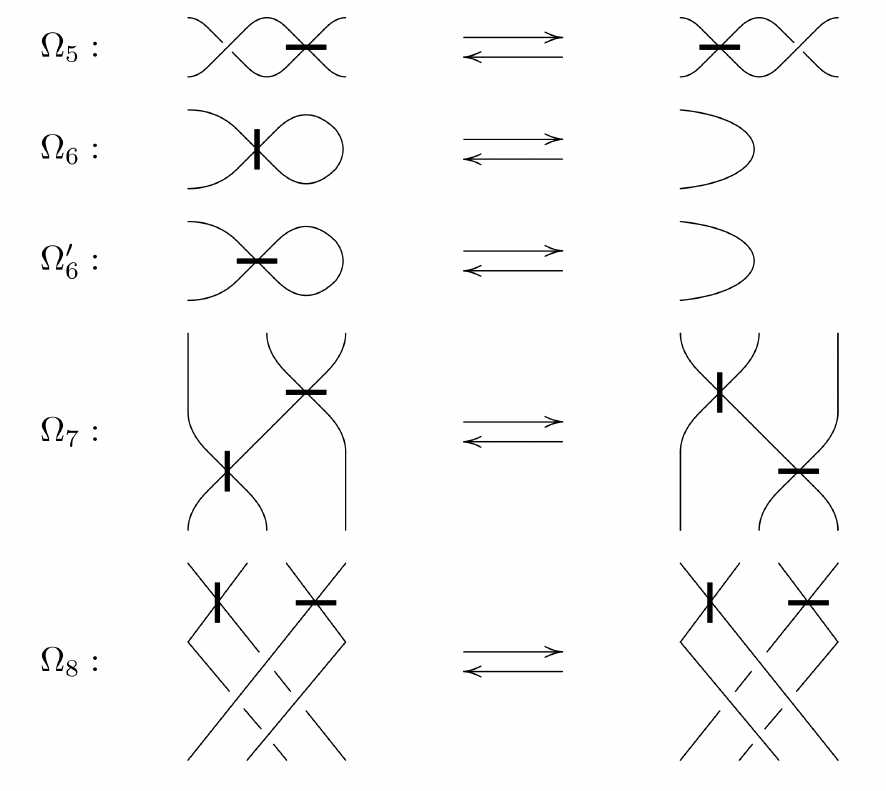}}\]
See \cite{KJL, KJL2} for more about these moves.

\section{\large\textbf{Bikei}}\label{B}

In this section we recall \textit{bikei} and the biquandle counting 
invariant for surface-links, including non-orientable surface-links.

\begin{definition}
Let $X$ be a set. A \textit{bikei structure} on $X$ consists of two 
binary operations $\otr,\utr$ on $X$ such that
\begin{itemize}
\item[(i)] For all $x\in X$, we have $x\utr x=x\otr x$,
\item[(ii)] For all $x,y\in X$ we have
\[\begin{array}{rcl}
(x\utr y)\utr y & = & x \\
(x\otr y)\otr y & = & x \\
x\utr (y\otr x) & = & x\utr y \\
x\otr (y\utr x) & = & x\otr y \\
\end{array}\]
and
\item[(iii)] For all $x,y,z\in X$, the \textit{exchange laws} are satisfied:
\[\begin{array}{rcl}
(x\utr y)\utr(z\utr y) & = & (x\utr z)\utr (y\otr z), \\
(x\utr y)\otr(z\utr y) & = & (x\otr z)\utr (y\otr z), \ \mathrm{and} \\
(x\otr y)\otr(z\otr y) & = & (x\otr z)\otr (y\utr z).
\end{array}\]
\end{itemize}
A bikei in which $x\otr y=x$ for all $x,y$ is called a \textit{kei} or
\textit{involutory quandle}.
\end{definition}

\begin{example}
Let $G$ be a group. The operations
\[ x\utr y=yx^{-1}y \quad \mathrm{and}\quad x\otr y=x\]
define a bikei structure on $X$ known as a \textit{Takasaki kei}, 
\textit{cyclic kei} or \textit{dihedral quandle} depending on context.
\end{example}

\begin{example}
Let $R$ be a commutative ring with identity and let $X$ be an $R$-module
with $t,r,s\in R$ such that $t^2=r^2=1$, $s(t+r)=0$ and $r=t+s$.
Then $X$ is a bikei known as an \textit{Alexander bikei} under the operations
\[x\utr y=tx+sy \quad\mathrm{and}\quad x\otr y= rx. \] 
\end{example}

\begin{example}
We can specify a bikei structure on a set $\{1,2,\dots,n\}$ with pair 
of operation tables (or more succinctly as a block matrix of indices).
For example the smallest non-trivial bikei has underlying set
$X=\{1,2\}$ with operations given by 
\[
\begin{array}{r|rr}
\utr & 1 & 2 \\ \hline
1 & 2 & 2\\
2 & 1 & 1
\end{array} \quad
\begin{array}{r|rr}
\otr & 1 & 2 \\ \hline
1 & 2 & 2\\
2 & 1 & 1
\end{array} \quad
\]
or more compactly by the block matrix 
\[\left[\begin{array}{rr|rr}
2 & 2 & 2 & 2 \\
1 & 1 & 1 & 1
\end{array}\right].\]
\end{example}

\begin{example}
Let $L$ be an unoriented surface-link represented by a marked graph
diagram $D$. The \textit{fundamental bikei} of $L$ has a presentation
with a generator for each semiarc and a crossing relation at each crossing,
with semiarcs meeting at a marked vertex having the same generator as shown
\[\includegraphics{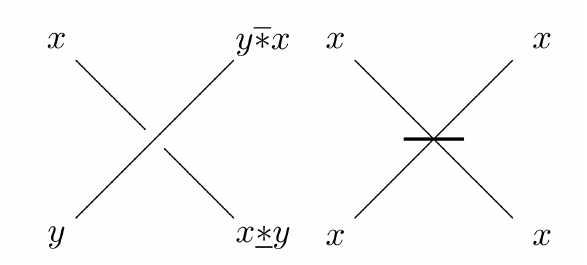}.\]
The elements of the fundamental bikei are equivalence classes of bikei words
in these generators (i.e. obtained recursively as $w\utr z$ or $w\otr z$
where generators are words) modulo the equivalence relation generated by
the bikei axioms and the crossing relations. 
\end{example}

The bikei axioms are chosen
precisely so that Yoshikawa moves (indeed, the first three moves, known
as the \textit{Reidemeister moves} in classical knot theory, are sufficient)
on diagrams induce Tietze moves on the presented fundamental bikei.  See 
\cite{EN} for more; here we show that Yoshikawa moves $\Omega_4$ through 
$\Omega_8$ do not impose any additional conditions.
\[\scalebox{0.85}{\raisebox{0.2in}{\includegraphics{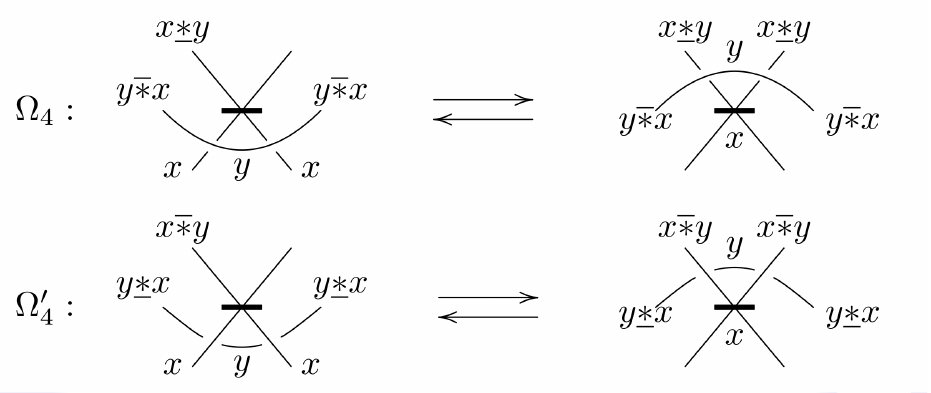}}
\quad \includegraphics{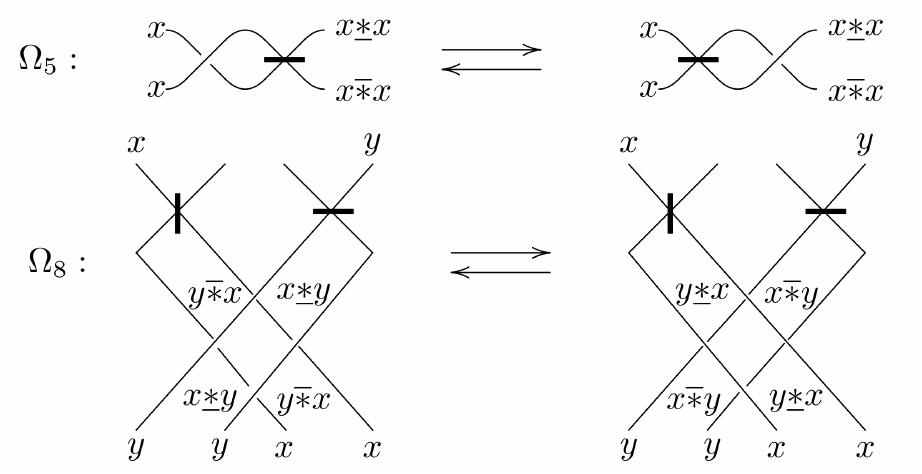}}\]

It follows that the fundamental bikei is an invariant of surface-links. However,
directly comparing objects specified by presentations is difficult; thus,
we want to indirectly compare these fundamental bikei by comparing their 
homsets onto a choice of finite bikei.

\begin{definition}
A map $f:X\to Y$ between bikei is a \textit{bikei homomorphism} if for all
$x,x'\in X$ we have
\[f(x\utr x')=f(x)\utr f(x')
\quad\mathrm{and}\quad
f(x\otr x')=f(x)\otr f(x').
\]
\end{definition}

\begin{definition}
Let $L$ be an unoriented surface-link and $X$ a bikei. 
A \textit{bikei coloring} of a diagram of $L$ by $X$ is an assignment of
elements in $X$ to the semiarcs of $L$ such that the crossing relations
are satisfied at every crossing.
\end{definition}

A bikei coloring uniquely determines and is determined by a bikei homomorphism 
$f:\mathcal{B}(L)\to X$. Thus a marked graph diagram gives us a way
to combinatorially (and visually) represent elements of the homset 
$\mathrm{Hom}(\mathcal{B},X)$ analogously to the way choosing a basis for
a vector space lets us concretely represent homsets in the vector space 
category. In particular, the cardinality of the homset is an integer-valued
invariant of unoriented surface-links known as the \textit{bikei counting
invariant}, denoted $\Phi_X^{\mathbb{Z}}(L)=|\mathrm{Hom}(\mathcal{B}(L),X)|$.

\begin{example}
The unknotted projective plane $2^{-1}_1$ with diagram
\[\includegraphics{yj-sn2021-8.pdf}\]
has no colorings by the bikei $X_1=\{1,2\}$ with operations
\[
\begin{array}{r|rr}
\utr & 1 & 2 \\ \hline
1 & 2 & 2\\
2 & 1 & 1
\end{array} \quad
\begin{array}{r|rr}
\otr & 1 & 2 \\ \hline
1 & 2 & 2\\
2 & 1 & 1
\end{array}
\]
since at the marked vertex all four colors must be equal but at the crossing
colors must change from 1 to 2 or 2 to 1; hence, we have 
$\Phi_{X_1}^{\mathbb{Z}}(2^{-1}_1)=0$. On the other hand, the same diagram 
has two colorings by the other bikei structure $X_2$ on the set of two elements,
\[
\begin{array}{r|rr}
\utr & 1 & 2 \\ \hline
1 & 1 & 1\\
2 & 2 & 2
\end{array} \quad
\begin{array}{r|rr}
\otr & 1 & 2 \\ \hline
1 & 1 & 1\\
2 & 2 & 2
\end{array}
\]
and we have $\Phi_{X_2}^{\mathbb{Z}}(2^{-1}_1)=2$.
\end{example}

\begin{example}\label{ex:8s}
Let $X$ be the bikei structure on the $\{1,2,3,4\}$ specified by the
operation tables
\[
\begin{array}{r|rrrr} 
\utr & 1 & 2 & 3 & 4 \\ \hline
1 & 3 & 1 & 3 & 1 \\
3 & 2 & 4 & 4 & 4 \\
3 & 1 & 3 & 1 & 3 \\
4 & 4 & 2 & 2 & 2 
\end{array}\quad
\begin{array}{r|rrrr} 
\otr & 1 & 2 & 3 & 4 \\ \hline
1 & 3 & 1 & 3 & 1 \\
2 & 4 & 4 & 2 & 4 \\
3 & 1 & 3 & 1 & 3 \\
4 & 2 & 2 & 4 & 2.
\end{array}
\]
The non-orientable surface-links $8^{-1,-1}_1$ and $9^{1,-2}_1$ are distinguished 
from each other by their counting invariants with respect to $X$, with
$\Phi_{X_2}^{\mathbb{Z}}(8^{-1,-1}_1)=0\ne 4=\Phi_{X}^{\mathbb{Z}}(9^{1,-2}_1)$.
\[\includegraphics{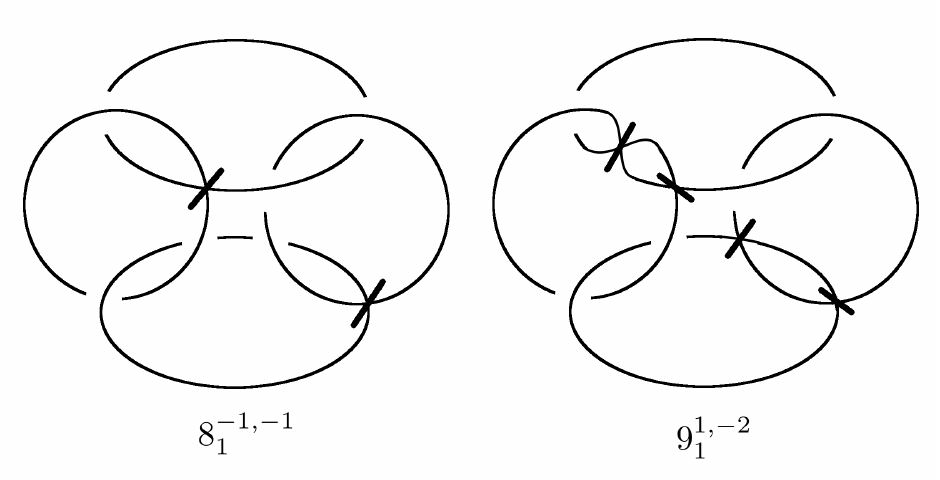}\]
\end{example}

\section{\large\textbf{Bikei Module Enhancements}}\label{BE}

We begin this section with a generalization of a definition from \cite{JN1}.

\begin{definition}
Let $X$ be a bikei and $R$ a commutative ring with identity.
A \textit{bikei module} is an assignment
of elements $t_{x,y}, r_{x,y}$ and $s_{x,y}$ in $R$ to each pair of elements
in $X$ such that for all $x,y,z\in X$ we have
\[\begin{array}{rcll}
t_{x,y}t_{x\utr y, y\otr x} & = & 1 & (0.i)\\
r_{x,y}r_{x\utr y, y\otr x} & = & 1 & (0.ii)\\
(t_{x,y}+r_{x,y})s_{x\utr y, y\otr x} & = & 0 & (0.iii) \\
t_{x,x}+s_{x,x} & = & r_{x,x}, & (i.i) \\
r_{y\otr x,z\otr x} r_{x,z} & = & r_{x\utr y, z\otr y}r_{y,z}, & (iii.i) \\
r_{x\utr z,y\utr z} t_{y,z} & = & t_{y\otr x, z\otr x}r_{x,y}, & (iii.ii) \\
r_{x\utr z,y\utr z} s_{y,z} & = & s_{y\otr x, z\otr x}r_{x,z}, & (iii.iii) \\
t_{x\utr z,y\utr z} t_{x,z} & = & t_{x\utr y, z\otr y}t_{x,y}, & (iii.iv) \\
s_{x\utr z,y\utr z} t_{y,z} & = & t_{x\utr y, z\otr y}s_{x,y}, & (iii.v) \\
t_{x\utr z,y\utr z} s_{x,z} +s_{x\utr z, y\utr z}s_{y,z} & = & s_{x\utr y, z\otr y}r_{y,z}.  & (iii.vi)
\end{array}\]
\end{definition}

We can specify a bikei module using a block matrix $[T|S|R]$ whose
blocks have $t_{j,k}$, $s_{j,k}$ and $r_{j,k}$ respectively in row $j$
column $k$.

\begin{example}
Using our \texttt{python} code, we found 512 bikei module structures 
on the bikei
\[
\begin{array}{r|rr}
\utr & 1 & 2 \\ \hline
1 & 2 & 2\\
2 & 1 & 1
\end{array} \quad
\begin{array}{r|rr}
\otr & 1 & 2 \\ \hline
1 & 2 & 2\\
2 & 1 & 1
\end{array}
\]
with coefficients in $R=\mathbb{Z}_8$, including for instance
\[\left[\begin{array}{rr|rr|rr}
3 & 7 & 4 & 0 & 7 & 5 \\
7 & 3 & 0 & 4 & 5 & 7
\end{array}\right].\]
\end{example}

The motivation is as follows: 
Given an element $f$ of the bikei homset represented by a bikei coloring of 
a marked graph diagram, we would like to define an invariant $R$-module
by modifying the Alexander bikei operations to have coefficients which
are functions of the bikei colors at each crossing. Assigning a module 
generator to each semiarc, we can picture these generators as ``beads''.
\[\includegraphics{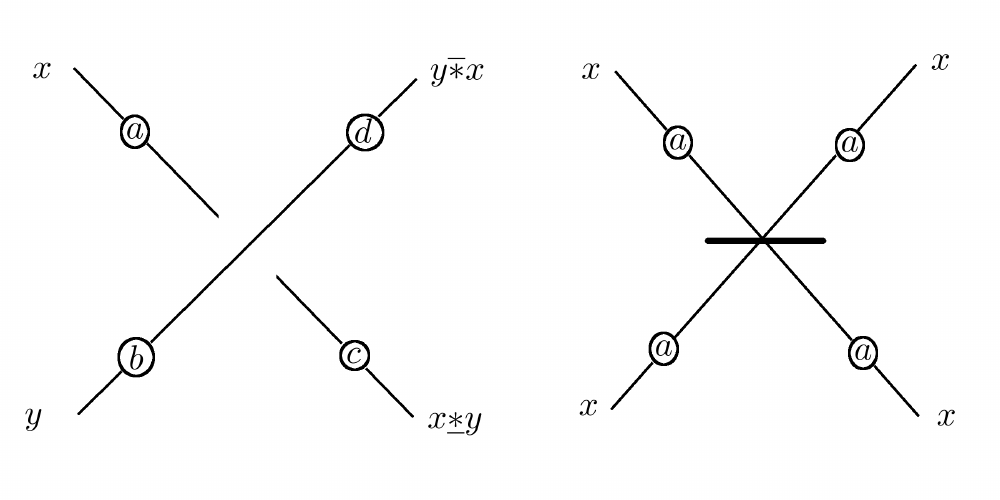}\]
We then define the module $M_f$ as the kernel of the coefficient matrix 
of the homogeneous determined by the system of linear equations including
\[\begin{array}{rcl}
0 & = & t_{x,y} a+s_{x,y} b -d \\
0 & = & r_{x,y} b -c
\end{array}\]
at each crossing. The elements of $M_f$ can be represented as 
\textit{bead colorings} of the $X$-colored diagrams, i.e. assignments of
elements of $R$ to each semiarc in the $X$-colored diagram satisfying the 
above-pictured condition at every crossing. The bikei module axioms are the 
conditions on the coefficients coming from the bikei-colored Yoshikawa moves 
(again, the first three, i.e. the Reidemeister moves, are sufficient) together 
with the 180-degree rotational symmetry required by the lack of orientation.

A choice of bikei module over a bikei $X$ with coefficients in $R$ associates
an invariant $R$-module to each $X$-coloring of a diagram $D$ of an unoriented 
surface-link $L$. More precisely, Yoshikawa moves induce $R$-module 
isomorphisms on these modules, and consequentially the multiset of such 
$R$-modules over the homset $\mathrm{Hom}(\mathcal{BK}(L),X)$ is an invariant
of unoriented surface-links. 

We illustrate the computation of the invariant with an example.
\begin{example}
Let $X=\{1,2\}$ have the trivial bikei structure, i.e. the structure 
specified by the operation tables
\[
\begin{array}{r|rrrr}
\utr & 1 & 2\\ \hline
1 & 1 & 1 \\
2 & 2 & 2 \\
\end{array}
\quad
\begin{array}{r|rr}
\otr & 1 & 2 \\ \hline
1 & 1 & 1  \\
2 & 2 & 2  \\
\end{array}
\]
and let $R=\mathbb{Z}_5$; then $X$ has $R$-module structures including
\[\left[\begin{array}{rr|rr|rr}
1 & 4 & 0 & 2 & 1 & 1 \\
1 & 4 & 0 & 0 & 4 & 4
\end{array}\right].\]
To compute the invariant for the unknotted projective plane $2^{-1}_1$, 
we note that there are two $X$-colorings of a marked graph diagram representing
$2^{-1}_1$ as shown:
\[\includegraphics{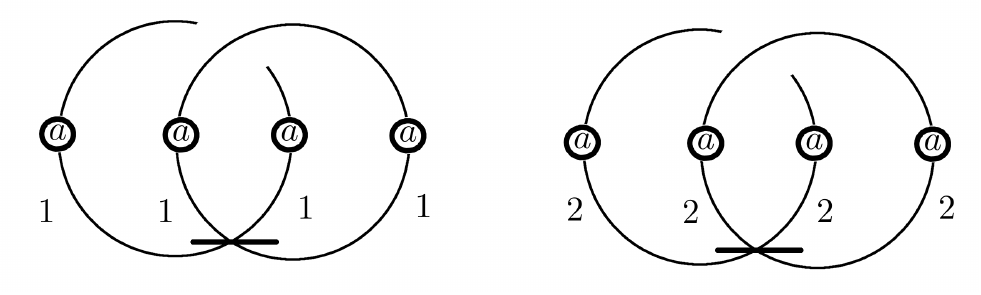}\]
These then give us the systems of bead coloring equations
\[
\begin{array}{rcl}
a & = & r_{11}a \\
a & = & t_{11}a+s_{11} a
\end{array}
\quad
\begin{array}{rcl}
a & = & r_{22}a \\
a & = & t_{22}a+s_{22} a
\end{array}
\]
which become
\[
\begin{array}{rcl}
a & = & 1a \\
a & = & 1a+0a
\end{array}
\quad
\begin{array}{rcl}
a & = & 4a \\
a & = & 1a+0 a
\end{array}
\]
so the associated $R$-modules are respectively the kernels of the matrices
$\left[\begin{array}{r}0 \\ 0\end{array}\right]$ and 
$\left[\begin{array}{r}1 \\ 0\end{array}\right]$,
and the invariant value is the multiset 
$\{(\mathbb{Z}_5)^0,(\mathbb{Z}_5)^1\}$. 

\end{example}

For ease of comparison, we can define a polynomial invariant from this
homset by summing a formal variable $u$ raised to the number of elements 
of each module in the case that $R$ is finite (that is, the number of bead 
colorings of each $X$-colored diagram) or raised to the rank of the module
in case $R$ is infinite. More formally, we have:

\begin{definition}
Let $X$ be a finite bikei, $L$ an unoriented surface-link represented by
a marked graph diagram $D$, $R$ a commutative ring with identity and $M$
a choice of $X$-module with coefficients in $R$. For each homset element
$f:\mathcal{BK}(L)\to X$, let $D_f$ be the corresponding $X$-coloring of
$D$, and let $M_f$ be the $R$-module of bead colorings of $D_f$ with $M$
coefficients. We then define the \textit{bikei module enhanced polynomial}
of $L$ to be
\[\Phi_X^{M}(L)=
\left\{\begin{array}{ll}
\sum_{f\in\mathrm{Hom}(\mathcal{BK}(L),X)} u^{|M_f|}& |R|\in \mathbb{N} \\
\sum_{f\in\mathrm{Hom}(\mathcal{BK}(L),X)} u^{\mathrm{Rank}(M_f)} & \mathrm{otherwise} 
\end{array}.\right.
\]
\end{definition}

\begin{theorem}
The polynomial $\Phi_X^{M}(L)$ is an invariant of surface-links.
\end{theorem}

\begin{proof}
The conditions defining bikei modules are the conditions required for 
invariance of $M_f$ under a generating set of unoriented Yoshikawa moves.
\end{proof}

\begin{example}\label{ex:proper}
Let $X=\{1,2,3,4\}$ have bikei structure defined by the operation tables
\[
\begin{array}{r|rrrr}
\utr & 1 & 2 & 3 & 4 \\ \hline
1 & 1 & 1 & 2 & 2 \\
2 & 2 & 2 & 1 & 1 \\
3 & 4 & 4 & 3 & 3 \\
4 & 3 & 3 & 4 & 4
\end{array}
\quad
\begin{array}{r|rrrr}
\otr & 1 & 2 & 3 & 4 \\ \hline
1 & 1 & 1 & 2 & 2 \\
2 & 2 & 2 & 1 & 1 \\
3 & 3 & 3 & 3 & 3 \\
4 & 4 & 4 & 4 & 4
\end{array}
\]
and let $R=\mathbb{Z}_3$. Then the matrix of $t,s,r$ coefficients
\[
\left[
\begin{array}{rrrr|rrrr|rrrr}
2 & 2 & 1 & 1 & 2 & 2 & 0 & 0 & 1 & 1 & 1 & 1 \\
2 & 2 & 1 & 1 & 2 & 2 & 0 & 0 & 1 & 1 & 1 & 1 \\
2 & 1 & 2 & 2 & 0 & 0 & 0 & 0 & 1 & 1 & 2 & 1 \\
1 & 2 & 1 & 1 & 0 & 0 & 0 & 0 & 1 & 1 & 2 & 1
\end{array}
\right]
\]
defines a bikei module over $X$ with $R$ coefficients. 

The non-orientable surface-links $8^{-1,-1}_1$ and $9^{1,-2}_1$ both have
bikei homsets of cardinality 8, so the links are not distinguished by the 
counting invariant alone with respect to this bikei. However, the
bikei module enhanced invariants with respect to the listed bikei module
are
\[\Phi_{X}^M(8^{-1,-1}_1)=5u+2u^3+u^9\]
and
\[\Phi_{X}^M(9^{1,-2}_1)=2u+4u^3+2u^9;\]
hence the surface-links are distinguished by the enhanced invariant.
In particular this example shows that the enhanced invariant is not
determined by the size of the homset and hence is a proper enhancement.
\end{example}

\begin{example}
We computed the bikei module enhanced invariant for the non-orientable 
surface-links with $ch$-index up to 10 as listed in \cite{Y} with respect 
to the bikei and bikei module in example \ref{ex:proper}; the results are in 
the table.

\[\begin{array}{r|l}
\Phi_X^M(L) & L \\ \hline
3u+u^3 & 2^{-1}_1, 10^{0,-2}_1, 10^{-1,-1}_1, 10^{-1,-1}_2 \\
5u+2u^3+u^9 & 8^{-1,-1}_1, 10^{-2,-2}_1 \\
2u+4u^3+2u^9 & 9^{1,-2}_1 \\
6u+4u^3+2u^9 & 7^{0,-2}_1.
\end{array}\]
\end{example}

\begin{example}
Let $X=\{1,2,3,4\}$ have bikei structure defined by the operation tables
\[
\begin{array}{r|rrrr}
\utr & 1 & 2 & 3 & 4 \\ \hline
1 & 1 & 1 & 2 & 2 \\
2 & 2 & 2 & 1 & 1 \\
3 & 4 & 4 & 3 & 3 \\
4 & 3 & 3 & 4 & 4
\end{array}
\quad
\begin{array}{r|rrrr}
\otr & 1 & 2 & 3 & 4 \\ \hline
1 & 1 & 1 & 2 & 2 \\
2 & 2 & 2 & 1 & 1 \\
3 & 3 & 3 & 3 & 3 \\
4 & 4 & 4 & 4 & 4
\end{array}
\]
and let $R=\mathbb{Z}_3$. Then the matrix of $t,s,r$ coefficients
\[
\left[
\begin{array}{rrrr|rrrr|rrrr}
1 & 2 & 1 & 1 & 1 & 0 & 0 & 0 & 2 & 2 & 2 & 2 \\
2 & 1 & 1 & 1 & 0 & 1 & 0 & 0 & 2 & 2 & 2 & 2 \\
2 & 1 & 2 & 2 & 0 & 0 & 0 & 0 & 1 & 1 & 2 & 1 \\
1 & 2 & 1 & 1 & 0 & 0 & 0 & 0 & 1 & 1 & 2 & 1
\end{array}
\right]
\]
defines a bikei module over $X$ with $R$ coefficients. 
We computed the bikei module enhanced invariant for the orientable and 
non-orientable 
surface-links with $ch$-index up to 10 as listed in \cite{Y}; the results are in 
the table. 



\[\begin{array}{r|l||r|l}
\Phi_X^M(L) & L _{ori} &\Phi_X^M(L) & L _{non-ori} \\ \hline \hline
4u^3 & 2_1,  10_1, 10_3           & 3u+u^3 & 2^{-1}_1, 10^{0,-2}_1, 10^{0,-2}_2, 10^{-1,-1}_1 \\
2u^3+10u^9 & 6^{0,1}_1            &  u+6u^3+u^9 & 8^{-1,-1}_1, 10^{-2,-2}_1\\
2u^3+2u^9 & 8_1, 9_1, 10_2, 10^1_1 & 6u^3+2u^9 & 9^{1,-2}_1\\
2u^3+6u^9 & 10^{1,1}_1            & 10u^3+2u^9 & 7^{0,-2}_1\\
2u^3+14u^9 & 9^{0,1}_1, 10^{0,1}_2 & \\
2u+4u^3+10u^9 & 9^{0,1}_2 & \\
4u^3+4u^9 & 8^{1,1}_1 & \\
6u^3+10u^9 & 10^{0,1}_1 & \\
18u^9+30u^{27} & 10^{0,0,1}_1. & 
\end{array}\]

\end{example}

\begin{remark}
In \cite{JKL}, non-orientable surface-links could not be distinguished using
ideal coset invariants, so these examples show that the method of this paper 
represents an improvement over the methods of \cite{JKL}.
\end{remark}

\begin{remark}
Though we have developed bikei modules as invariants of surface-links, 
the invariants they define also apply to unoriented classical and virtual knots
and links, which may be regarded as trivial cobordisms between two copies
of the classical or virtual knot or link. 
In particular, every bikei module is a biquandle module as defined in 
\cite{JN1} satisfying some extra conditions; hence the set of bikei modules
of a bikei $X$ is a subset of the set of biquandle modules of $X$. The 
biquandle modules over $X$ which are not bikei modules are those that
may be sensitive to orientation; thus, to find invariants which can 
potentially distinguish knots from their reverses, we can use biquandle 
module invariants using biquandle modules which are not bikei modules.
\end{remark}

\section{\large\textbf{Questions and Future Directions}}\label{Q}

We have considered only bikei modules over finite rings of the form 
$\mathbb{Z}_n$ for simplicity of computation using \texttt{python}. 
Bikei modules over infinite rings, especially polynomial rings, 
are of great interest as generalizations of the Alexander invariant
for these non-orientable surface-links. Our method of exhaustive computer 
search over finite rings obviously does not work for these cases, so
developing other methods of finding bikei modules is of definite interest.

What is the relationship between bikei modules and other bikei invariants?
Do bikei modules have an interpretation as bikei extensions analogous to
abelian extensions of quandles by cocycles as in \cite{CENS}?

A more computational question: we searched many examples of bikei modules
over finite rings but could not distinguish the surface-links $8^{-1,-1}_1$ 
and $10^{-2,-2}_1$. We 
conjecture that this is a consequence of our computational limitation to
small cardinality bikei and rings, and ask what is the smallest
bikei $X$ and bikei module $M$ over a finite ring $R$ such that
\[\Phi_{X}^M(8^{-1,-1}_1)\ne\Phi_{X}^M(10^{-2,-2}_1).\]

\bibliography{yj-sn}{}
\bibliographystyle{abbrv}

\noindent
\textsc{Institute for Mathematical Sciences \\
Pusan National University \\
Busan 46241, Korea.} \\

\noindent
\textsc{Department of Mathematical Sciences \\
Claremont McKenna College \\
850 Columbia Ave. \\
Claremont, CA 91711}

\end{document}